\documentclass[10pt]{article}
\usepackage[english]{babel}
\usepackage{amssymb,amsthm,amsmath,amsfonts,mathrsfs}
\usepackage{tikz}
\usepackage{bbold}
\usepackage{latexsym}
\usepackage{ifthen, bigstrut}
\usepackage{accents}
\usepackage[english]{babel}
\usepackage{hyperref}
\usepackage{breakurl}
\usepackage{ifthen, bigstrut}

\usepackage{enumerate}

\newtheorem{theorem}{Theorem}

\newtheorem{corollary}[theorem]{Corollary}
\newtheorem{lemma}[theorem]{Lemma}

\newcommand{\rank}{{\rm rank}}

\newcommand{\R}{\mathcal{R}}

\newcommand{\D}{\ensuremath{\mathcal{D}}}
\newcommand{\E}{\ensuremath{\mathcal{E}}}
\newcommand{\U}{\ensuremath{\mathcal{U}}}
\newcommand{\M}{\ensuremath{\mathcal{M}}}

\newcommand{\GL}{{G^\mathcal{L}}}
\newcommand{\GR}{{G^\mathcal{R}}}

\newcommand{\HL}{H^\mathcal{L}}
\newcommand{\HR}{H^\mathcal{R}}
\newcommand{\XL}{X^\mathcal{L}}

\newcommand{\BL}{{B^\mathcal{L}}}
\newcommand{\BR}{{B^\mathcal{R}}}

\renewcommand{\ge}{\geq}
\renewcommand{\le}{\leq}
\newcommand{\bs}[1]{\ensuremath{\boldsymbol{#1}}}
\newcommand{\Conj}[1]{\overline{#1}}

\newcommand{\PM}{perfect murder}
\newcommand{\Pm}{Perfect murder}

\renewcommand{\geq}{\geqslant}
\renewcommand{\leq}{\leqslant}

\theoremstyle{definition}
\newtheorem{definition}[theorem]{Definition}

\newtheorem{observation}[theorem]{Observation}

\newtheorem{example}[theorem]{Example}

\usepackage{xcolor}

\begin{document}
\begin{center}
{\large {\bf Progress on mis\`ere dead ends: game comparison, canonical form, and conjugate inverses}}
\vspace{1 cm}

URBAN LARSSON\footnote{Supported in part by the Aly Kaufman Fellowship}\\
 {\small {\em  Industrial Engineering and Management, Technion - Israel Institute of Technology, Israel.}\\{\em urban031@gmail.com}}
 
 \vspace{.5 cm}
REBECCA MILLEY \\
{\small {\em Computational Mathematics, Grenfell Campus, Memorial University, Canada.}\\{\em rmilley@grenfell.mun.ca  }}

 \vspace{.5 cm}
%
RICHARD NOWAKOWSKI\footnote{Supported in part by the Natural Sciences and Engineering Research Council of Canada.}\\
{\small {\em Dept of Mathematics and Statistics, Dalhousie University, Canada.}\\{\em rjn@mathstat.dal.ca  }}

 \vspace{.5 cm}
GABRIEL RENAULT\footnote{Supported by the ANR-14-CE25-0006 project of the French National Research Agency.}\\ {\small {\em gabriel.renault@ens-lyon.org}}

 \vspace{.5 cm}
%
CARLOS SANTOS\footnote{This work was partially funded by Funda\c c\~ao para a Ci\^encia e a Tecnologia through the project UID/MAT/04721/2013.}\\
{\small {\em Center for Functional Analysis, Linear Structures and Applications, University of Lisbon, Portugal.}\\{\em cmfsantos@fc.ul.pt  }}\\
\end{center}

\vskip 30pt

\begin{abstract}
This paper addresses several significant gaps in the theory of restricted mis\`ere play (Plambeck, Seigel 2008), primarily in the well-studied universe of dead-ending games, $\mathcal{E}$ (Milley, Renault 2013); if a player run out of moves in $X\in \mathcal E$, then they can never move again in any follower of $X$. A \emph{universe of games} is a class of games which is closed under disjunctive sum, taking options and conjugates. 
We use novel results from absolute combinatorial game theory (Larsson, Nowakowski, Santos 2017) to show that $\mathcal{E}$ and the universe $\mathcal{D}\subset \mathcal{E}$ of dicot games (either both, or none of the players can move) have `options only'  test for comparison of games, and this in turn is used to define unique reduced games (canonical forms) in $\mathcal{E}$. We develop the reductions for $\mathcal{E}$ by extending analogous work for $\mathcal{D}$, in particular by solving the problem of \emph{reversibility through ends} in the larger universe. 
Finally, by using the defined canonical forms in $\mathcal{E}$ and  $\mathcal{D}$, we prove that both of these universes, as well as the subuniverse of impartial games, have the \emph{conjugate property}: every inverse game is obtained by swapping sides of the players.
\end{abstract}
\section{Introduction}
In this paper, we develop theory of mis\`ere-play games \cite{Si} in the dicot and dead-ending settings; see \cite{DorbecRSS} and \cite{MilleyRenault} respectively. The major contribution is the complete solution of game reduction for the monoid of dead-ending games. Some recent results on constructive game comparison \cite{LarssonNS2016A, LNS2017} have been adapted. Moreover, we report, in both dicot and dead-ending settings, that if a game has an inverse, then it is the conjugate. 

\subsection{Universes of games}
Let $\M$ be the class of all short 2-player mis\`ere-play games \cite{Si}. A short game $G\in \M$ is a rooted tree, with finitely many nodes and two kind of move edges, left- or right-slanting. The players Left and Right alternate in moving; Left (Right) plays along left- (right-) slanting edges. The root is the starting position, and the leaves are the terminal positions. 

In mis\`ere-play, a player who cannot move wins. Since mis\`ere-play theory is hard, at least in comparison with the normal-play ditto \cite{onag} (a player who cannot move loses), recent work has focused on studying well suited restrictions of $\M$ (see \cite{MilleyRenault} for a survey). The restrictions relevant for this work are the classes of dicot and dead-ending games. Many well-studied rulesets, including {\sc Domineering} and {\sc Hackenbush}, have the {\em dead-ending} property: informally, once a player runs out of moves in a game, that player will never again have a move in this game. We require some definitions.

In an equivalent recursive representation, a game $G = \left\{\GL\mid \GR\right\}$ is given by its set of Left and Right options, $\GL$ and $\GR$, respectively, where $G^L\in \GL$ ($G^R\in \GR$) is a typical Left (Right) option of $G$. The \emph{followers} of $G$ are the nodes in the game tree, beginning with the game itself, then the options, then the options of the options, etc.

A game is a \emph{dicot} if, for each follower, either none of the players can move, or both players can move. 
 
A game $G$ is a \emph{Left-end} if it has no Left option, and we write $G = \left\{\;\mid \GR \right\}$; a game is a \emph{dead Left-end} if each follower is a Left-end, and analogously for Right. A \emph{dead end} is a dead Left-end or a dead Right-end. 

A game is \emph{dead-ending} if each end-follower is a \emph{dead end}.

The class of all dicot games is $\D$ and the class of all dead-ending games is $\E$. Note that $\D$ is the restriction of $\E$, where each dead end is the empty game. Both classes satisfy some important closure properties. First we require two more central concepts. 

In a disjunctive sum of games, the current player chooses exactly one of the game components and plays in it, while the other components remain the same. That is, the disjunctive sum of the games $G$ and $H$ is $$G + H = \left\{\GL +H, G+\HL\mid \GR+H, G+\HR\right\},$$ where $\GL + H = \left\{G^L+H: G^L\in \GL\right\}$, etc. 

A conjugate of a game $G$ is $\Conj{G} = \left\{\Conj{\GR}\mid\Conj{\GL}\right\}$, where $\Conj{\GL}=\left\{\Conj{G^L}: G^L\in \GL\right\}$, etc. 

\begin{definition} \label{def:universe}
A {\em universe}, $\U$, is a non-empty class of games, which satisfies the following properties:
	\begin{enumerate}
	\item option closure: if $G\in \U$ and $G'$ is an option of $G$ then $G'\in \U$;
	\item disjunctive sum closure: if $G,H\in \U$ then $G+H\in \U$;
	\item conjugate closure: if $G\in \U$ then $\Conj{G}\in \U.$
	\end{enumerate}
\end{definition}

It is easy to see that the sum of two dicots is a dicot, all followers of a dicot are dicots, and the conjugate of a dicot is dicot. The same is true for dead-ending games. Thus, both these classes of games are universes of games.\footnote{A subuniverse of $\mathcal{M}$ can be defined by taking the additive closure of all positions occurring under some rule set (e.g., the universe of  domineering positions); by taking the additive closure of all positions occurring in a given game tree (see \cite{PlambS2008} for examples); or by taking all positions that satisfy some property (e.g., the universe of $\mathcal{I}$ of impartial games, the universe  $\mathcal{D}$ of dicot games, or the universe  $\mathcal{E}$ of dead-ending games).} 

\subsection{Optimal play}

The possible (mis\`ere) \emph{results} of a game are L (Left wins) and R (Right wins); by convention, they are totally ordered with $\rm L > \rm R$.  
 The \emph{left-} and \emph{right-outcome}, in \emph{optimal play} from both players, of a mis\`ere-play game $G$ is 
 
 \[  o_L(G)= \left\{
\begin{array}{ll}
      {\rm L}, & \textrm{ if $\GL =\emptyset$;} \\
	{\max o_R(G^L)}, & \textrm{otherwise,} \\
\end{array} 
\right. \]
\[o_R(G)= \left\{
\begin{array}{ll}
      {\rm R}, & \textrm{ if $\GR =\emptyset$;} \\
	{\min o_L(G^R)}, & \textrm{otherwise,} 
\end{array} 
\right. \]
respectively. 

That is, $o_L(G)=\rm L$ if and only if Left wins $G$, playing first, and so on.

The elements of the cartesian product of the set $\left\{\rm L,\rm R\right\}$ with itself  
partition the possible {\em outcomes} of a game: 

\[  o(G)= \left\{
\begin{array}{ll}
      \mathscr{L}, & \textrm{if } (o_L(G),o_R(G)) = (\rm L,\rm L); \\
      \mathscr{N}, & \textrm{if } (o_L(G),o_R(G)) = (\rm L,\rm R); \\
      \mathscr{P}, & \textrm{if } (o_L(G),o_R(G)) = (\rm R,\rm L); \\
      \mathscr{R}, & \textrm{if } (o_L(G),o_R(G)) = (\rm R,\rm R).
\end{array}\right. \]
The outcomes inherit a partial order from the results with $\mathscr{L}> \mathscr{N} > \mathscr{R}$, $\mathscr{L} > \mathscr{P} > \mathscr{R}$, and where $\mathscr{N}$ and $\mathscr{P}$
are incomparable.

Suppose that Right receives the following offer just when he is about to start a game: for $G\in \E$, an arbitrary game of rank $k>0$, he may, if he wishes, \emph{design} a Left-end $E$ to be played in disjunctive sum with $G$. Note that the outcome of a non-trivial Left-end is Left wins, so does the challenge really make sense? In fact, there is a class of games, with a modest number of nodes, about twice the rank of $G$, which makes the offer very attractive. Let us present the \emph{perfect murder} of dead-ending mis\`ere play, a good tool for any player adventurous enough to take on a challenge, where in general `moves are good', except at the very end. In a sense, Right wants to use a perfect murder to maximize the chance that Left plays last in an arbitrary game $G$ of specified rank. 

\begin{definition}\label{def:murder}
The \emph{\PM} of rank $n$, $M_n\in \mathcal{E}$, is recursively defined by:
\[  M_n= \left\{
\begin{array}{ll}
      \boldsymbol 0, & \textrm{if } n=0; \\
	\left\{\; \mid \boldsymbol 0,M_{n-1}\right\}, & \textrm{if } n>0. \\
\end{array} 
\right. \]
\end{definition}

Thus,  $M_0 = \boldsymbol 0$, $M_1 = \left\{\,\mid \boldsymbol 0\right\}$, $M_2 = \left\{\, \mid \boldsymbol 0,\left\{\, \mid \boldsymbol 0\right\}\right\}$, and so on. Figure \ref{fig:murder} shows perfect murders of rank up to 4. 
Although, one may think of the perfect murders as a kind of `pass tool' for Right, their outcomes in isolation are `Left wins'.
\begin{observation}
The outcomes of the perfect murders are $o(M_0)=\mathscr N$, and $o(M_k)=\mathscr L$, for $k>0$.
\end{observation}

\begin{figure}
\begin{center}

\begin{tikzpicture}
[scale=0.75, auto=center]
\tikzstyle{every node}=[font=\large]

\tikzstyle{blackedge}=[draw,line width=1pt,-,black!100]

\tikzstyle{blueedge}=[draw,line width=1pt,-,blue!100]
\tikzstyle{rededge}=[draw,line width=1pt,-,red!100]

\node[circle, draw, scale=0.5, fill=black, label=$M_0$] (n0) at (0.4,0) {};

\node[circle, draw, scale=0.5, fill=black, label=$M_1$] (n1a) at (3,0) {};
\node[circle, draw, scale=0.5, fill=black ] (n1b) at (3.5,-1) {};
\draw[blackedge] (n1a)--(n1b);

\node[circle, draw, scale=0.5, fill=black,label=$M_2$ ] (n2a) at (6,0) {};
\node[circle, draw,scale=0.5,  fill=black] (n2b) at (6.5,-1) {};
\node[circle, draw, scale=0.5, fill=black] (n2c) at (7,-1) {};
\node[circle, draw,scale=0.5,  fill=black] (n2d) at (8,-2) {};
\draw[blackedge] (n2a)--(n2b);
\draw[blackedge] (n2a)--(n2c)--(n2d);

\node[circle, draw, scale=0.5, fill=black,label=$M_3$ ] (n3) at (9,0) {};
\node[circle, draw, scale=0.5, fill=black ] (n3a1) at (9.5,-1) {};
\node[circle, draw, scale=0.5, fill=black ] (n3a) at (10,-1) {};
\node[circle, draw,scale=0.5,  fill=black] (n3b) at (10.5,-2) {};
\node[circle, draw, scale=0.5, fill=black] (n3c) at (11,-2) {};
\node[circle, draw,scale=0.5,  fill=black] (n3d) at (12,-3) {};
\draw[blackedge] (n3a1)--(n3);
\draw[blackedge] (n3a)--(n3);
\draw[blackedge] (n3a)--(n3b);
\draw[blackedge] (n3a)--(n3c)--(n3d);

\node[circle, draw, scale=0.5, fill=black,label=$M_4$ ] (n4t) at (12,0) {};
\node[circle, draw, scale=0.5, fill=black ] (n4t1) at (12.5,-1) {};

\node[circle, draw, scale=0.5, fill=black ] (n4) at (13,-1) {};
\node[circle, draw, scale=0.5, fill=black ] (n4a1) at (13.5,-2) {};
\node[circle, draw, scale=0.5, fill=black ] (n4a) at (14,-2) {};
\node[circle, draw,scale=0.5,  fill=black] (n4b) at (14.5,-3) {};
\node[circle, draw, scale=0.5, fill=black] (n4c) at (15,-3) {};
\node[circle, draw,scale=0.5,  fill=black] (n4d) at (16,-4) {};

\draw[blackedge] (n4t)--(n4t1);
\draw[blackedge] (n4t)--(n4);
\draw[blackedge] (n4a)--(n4b);
\draw[blackedge] (n4a1)--(n4);
\draw[blackedge] (n4a)--(n4);
\draw[blackedge] (n4a)--(n4b);
\draw[blackedge] (n4a)--(n4c)--(n4d);

\end{tikzpicture}\caption{\Pm\ games of rank 0 to 4.}\label{fig:murder}
\end{center}
\end{figure}

The important property is revealed in Theorem~\ref{thm:pm} below. (The proofs are simpler here than in the case of Guaranteed scoring games \cite{LarssonNPS}, because now we have only 4 outcomes, so instead of Left- and Right-outcomes we can say Left wins, etc., and we make the logic explicit in the below proof.)

\begin{lemma} \label{lem:pm}
For all $n>0$, $M_n\geq_\mathcal{E}M_{n+1}$.
\end{lemma}

\begin{proof}
We prove that $o(M_n+X)\geq o(M_{n+1}+X)$, for all $n>0$ and for all $X \in \mathcal{E}$. 

Consider first $X=\bs 0$, and observe that $o(M_n) = \mathscr{L} = o(M_{n+1})$, for all $n > 0$. 

For $X\ne \bs 0$, suppose that $o(M_n+X) \ne \mathscr{L}$ (for otherwise we are done). 

The first case is that $o_L(M_n+X) = \rm R$, i.e. Right wins $M_n+X$ going second. Left playing first can only move in $X$, and if she loses on $M_n+X^L$, then by induction she loses on $M_{n+1}+X^L$ (if she does not have a move, then we are done as in the base case). That is, $o_L(M_{n+1}+X) = \rm R$. 

The second case is that $o_R(M_n+X) = \rm R$, i.e. Right wins $M_n + X$ playing first.

(1) If $\boldsymbol 0+X$ is a winning move for Right in $M_n+X$, then the same move is available from $M_{n+1}+X$.

(2) If  $M_{n-1}+X$ is a winning move for Right in $M_n+X$ (which is to $\boldsymbol 0+X$ if $n = 1$), then by induction, Right wins moving from $M_{n+1}+X$ to $M_{n}+X$.  

(3) If $M_{n}+X^R$ is a winning move, then by induction, Right wins moving from $M_{n+1}+X$ to $M_{n+1}+X^R$.
\end{proof}

\begin{theorem} \label{thm:pm}
If  $G$ is a Left-end with $\rank(G) = k > 0$, then $G\geq_\mathcal{E}M_n$, for all $n\geq k$.
\end{theorem}

\begin{proof}
Let $G\in \mathcal{E}$ be a fixed Left-end of rank $k > 0$. By Lemma \ref{lem:pm}, it suffices to show $G\geq_\mathcal{E}M_k$.

We must prove, for all $X\in \E$, $o(G+X)\geq  o(M_k+X)$. 
For $X = \boldsymbol 0$, since $G$ is a non-zero Left-end, then $o(G) = \mathscr{L} = o(M_k)$.  For $\rank(X) > 0$, we prove that
\begin{itemize}
\item[(i)] if $o_L(M_k+X)=\rm L$, then $o_L(G+X)=\rm L$; 
\item[(ii)]  if $o_R(M_k+X)=\rm L$, then $o_R(G+X)=\rm L$; 
\item[(iii)] if $o_L(G+X)=\rm R$, then $o_L(M_k+X)=\rm R$; 
\item[(iv)]  if $o_R(G+X)=\rm R$, then $o_R(M_k+X)=\rm R$. 
\end{itemize}
Notice that, since there are only 2 results, then (i) and (iii) is logically equivalent, and so is (ii) and (iv). Item (i) is true by induction on the rank of $X$, since there are no Left options in $M_k$ and $G$. 

For (iv), there are three possibilities. 

(1) There is a Right option $G^R = \boldsymbol 0$ such that $o_L(\boldsymbol 0 + X) = \rm R$. By construction $\boldsymbol 0\in M_k^\mathcal R$, so $o_R(M_k+X) \le o_L(\boldsymbol 0 + X) = \rm R$.

(2) If $G^R+X$ is a winning move, with $G^R\ne \bs 0$, then, since $\rank (G^R)\le k-1$, by induction $G^R\geq_\mathcal{E}M_{k-1}$, so $M_{k-1}+X$ is a  winning move for Right in $M_{k}+X$.

(3) If $G+X^R$ is a winning move, then by induction, so is $M_k+X^R$. 
\end{proof}

The \emph{strong Left outcome} is essentially the same as the outcome of $G$ when Right is allowed to `pass' at every turn.

\begin{definition}\label{def:so}
The \emph{strong Left-outcome} and  \emph{strong Right-outcome} of $G\in \E$ is  
\[ \begin{array}{ccc}
 \hat{o}_L(G)=\underbrace{\min}_{\emph{Left-end } X } \left\{o_L(G+X)\right\},  \\\\
 \hat{o}_R(G)=\underbrace{\max}_{\emph{Right-end } Y } \left\{o_R(G+Y)\right\},
\end{array} \]
respectively. 
\end{definition}

This definition can be simplified, by using the perfect murders.

\begin{theorem}\label{thm:so}
Let $k=\rank (G)$, with $G\in \E$. Then 
 $\hat{o}_L(G) = \min\left\{o_L(G), o_L(G + M_{k-1})\right\}$ and 
$ \hat{o}_R(G) = \max\left\{o_R(G), o_R(G + \Conj{M_{k-1}})\right\}$. 
\end{theorem}
\begin{proof}
Combine Theorem~\ref{thm:pm} with Definition~\ref{def:so}.
\end{proof}

\begin{example}
If $G=\left\{*\mid \boldsymbol 0\right\}$ then $o_L(G)= \rm L$, because $*$ is a winning move. But $\hat{o}_L(G)=\rm R$, because Right wins if Left plays first on  $G+M_1 = G-\boldsymbol 1$, and note that $\rank (G)=2$. If, on the other hand, $G=\left\{-\boldsymbol 1\mid \boldsymbol 0\right\}$, then both $o_L(G)={L}$ and $\hat{o}_L(G)={L}$. Right cannot win the game $G + X$, for any Left-end $X$, if Left starts, because Left has no more moves.
\end{example}

We pair the strong Left- and Right-outcomes of a game to give the \emph{strong outcome} of the game. 

\begin{definition}\label{def:stronginequality}
The \emph{strong outcome} of $G\in \E$ is
\[\hat{o}(G)
= \left\{
\begin{array}{ll}
      \mathscr{L}, & \textrm{if } (\hat{o}_L(G),\hat{o}_R(G)) = (\rm L,\rm L); \\
      \mathscr{N}, & \textrm{if } (\hat{o}_L(G),\hat{o}_R(G)) = (\rm L,\rm R); \\
      \mathscr{P}, & \textrm{if } (\hat{o}_L(G),\hat{o}_R(G)) = (\rm R,\rm L); \\
      \mathscr{R}, & \textrm{if } (\hat{o}_L(G),\hat{o}_R(G)) = (\rm R,\rm R).
\end{array}\right. \]
\end{definition}

\begin{observation}
Let $E$ be a non-zero dead Left-end.  Then $\hat{o}(E) = \mathscr{L}$, because if Left goes first then she has no move and wins, and if Right goes first then the position is still a Left-end (because the original position was a dead end) and so Left also wins going second.
\end{observation}

\section{Game comparison}
As described earlier, general mis\`ere analysis is in general hard. 
A breakthrough in the study of mis\`ere impartial games occurred when Plambeck and Siegel (\cite{Plamb2005, PlambS2008}) suggested weakened  equality and inequality relations in order to compare games in restricted classes of games. Later Milley, Renault et. al. generalized this work to the setting of partizan games; see \cite{MilleyRenault2017} for a survey. 

For example, we can talk about two dicot games being equivalent \emph{modulo dicots}, even if they are not equivalent in $\mathcal{M}$. 

In most cases, the tested games will also belong to the universe in which they are being compared, but this is not necessary.

\begin{definition}\label{def:eq} 
For a universe $\mathcal{U}\in\M$ and games $G,H\in\M $, then $G \geq_{\U} H  \textrm{ if } o(G+X)\geq o(H+X) \textrm{ for all games } X \in \mathcal{U}.$
The games are equivalent modulo $\U$, i.e. $G\equiv_\U H$, if $G \geq_{\U} H$ and $H \geq_{\U} G$.
\end{definition}
The word {\em indistinguishable} is sometimes used instead of {\em equivalent}, and if $G\not \equiv_{\U} H$  then $G$ and $H$ are said to be {\em distinguishable} modulo $\mathcal{U}$.  In this case there exists a game $X\in \mathcal{U}$ such that $o(G+X)\not = o(H+X)$, and we say that $X$ {\em distinguishes} $G$ and $H$. 
Notice that $G \equiv_{\U}  H$ implies $G \equiv_{\mathcal{V}} H$ for any subuniverse $\mathcal{V} \subseteq \mathcal{U}$, but in general, games can be equivalent in the smaller universe and distinguishable in the larger. 

It was shown that in $\mathcal{M}$, if $G\ne \bs 0$ then $G+\Conj{G} \not\equiv \bs 0$ \cite{MesdaO2007}, and moreover it is known that here $\bs 0$ is the only game equivalent to $\bs 0$. In $\mathcal{D}$ and $\mathcal{E}$, we can easily find games $G, H\not\equiv \bs 0$, such that $G+H\equiv \bs 0$, and thus the equivalence class of all neutral elements is `large'. We say $G$ is \textit{invertible} modulo $\U$ if there exists $H\in \U$ such that $G+H \equiv_{\U} 0$. 

Let us define the mis\`ere `integers'; $\bs 0$ is the empty game, $\bs 1=\{\bs 0\mid \; \}$,  $-\bs 1=\{\; \mid \bs 0\}$, $\bs 2=\{\bs 1\mid \; \}$,  $-\bs 2=\{\; \mid \bs 1\}$, and so on. In $\M$, these numbers do not have nice arithmetic properties. It is true that $\bs 1 +\bs 1 \equiv \bs 2$, but, for example $\bs 1 - \bs 1\not \equiv \bs 0$. We benefit by restricting the universe to the dead-ending games; namely, all integers belong to the class (they do not belong in dicot), and moreover, we have the standard arithmetics, so that for example $\bs 3 - \bs 7\equiv_\E \bs -\bs 4$. This follows because in more generality it was shown \cite{MilleyRenault} that all dead ends are invertible, and their inverses are the conjugates. 

Observe, if $G = \left\{\bs 0,* \mid \bs 0,*\right\}$, then the conjugate is $\Conj{G} = G$ and $G+\Conj{G}\not \equiv \bs0$ in any mis\`ere universe, because $o(G+\Conj{G}) = \mathscr P\ne  \mathscr N = o(\bs 0) $. One of our main results (Theorem \ref{Conjugate}, Section~\ref{sec:conjugate}) shows that in both $\mathcal{D}$ and $\mathcal{E}$  all inverses are the conjugates, which we call the `conjugate property'. 

We use the following result, adapted from \cite{LarssonNPS} (and with identical proof), to prove this property.

\begin{lemma}[\cite{LarssonNPS} A Cancellative Property]\label{lem:dis}
For any universe $\U\subseteq\M$ and any games $G, H, J \in \U$, if $G\geq_\mathcal{U} H$ then
$G+J\geq_\mathcal{U} H+J$. Moreover, the \emph{only if} directions holds, if $J$ is invertible. 
\end{lemma}

The final lemma, due to \cite{MesdaO2007}, is analogous to the `hand-tying' principle from normal play. The only difference is that in mis\`ere play, Left must have at least one option before additional options are guaranteed to give a position the same or better for Left; this is in analogy with the greediness principle in guaranteed scoring play \cite{LarssonNPS}. 

\begin{lemma}[\cite{MesdaO2007} A Mis\`ere Hand-tying Principle.]\label{lem:greediness}
Let $G\in\mathcal{M}$. If  $|\GL|\ge 1$  then for any $A\in\mathcal{M}$,
$ \left\{ \GL\cup \left\{A\right\}\mid {\GR}\right\}\geq_\mathcal{U} G$.
\end{lemma}

\subsection{Subordinate Game Comparison}\label{sec:gamecomparison} 
 Whereas in normal-play $G\geq H$ if and only if Left wins $G-H$ playing second, in mis\`ere-play, we do not generally have this shortcut. In general, to show $G\geq_\U H$, we must compare $o(G+X)$ and $o(H+X)$ for {\em all} games $X$ in $\U$. In this section, we develop a subordinate game comparison, which does not require consideration of all possible $X$, for the universes $\mathcal{D}$ and $\mathcal{E}$. Our construction is a specific application of absolute CGT \cite{LarssonNS2016A}; this is a major tool (the proof is highly non-trivial and uses an adjoint operation, down-linked relation, and other concepts from Siegel, Ettinger et al. \cite{Siege, Ettinger}), and we will use it to prove our main results in Sections \ref{sec:reductions} and \ref{sec:conjugate}. We require one more definition.
 
\begin{definition}
A universe $\U\subset \M$ is \emph{absolute} if it is parental and dense. A universe $\U$ is \emph{parental}, if for any non-empty sets $S,T\subset \U$, then $\{S\mid T\}\in \U$, and it is \emph{dense}, if, for any outcome $x$, and any game $G$, then there is a game $H$ such that $o(G+H)=x$. 
\end{definition}

\begin{theorem}[\cite{LarssonNS2016A} Common Normal Part]\label{thm:commonnormalpart}
Consider an absolute universe $\U\subset \M$, and let $G,H\in \mathcal{U}$. If $G\geq_\mathcal{U} H$, then
\begin{enumerate}
  \item  $\forall H^L\in H^{\mathcal{L}}, \mbox{either }
  \exists G^L\in G^{\mathcal{L}}\!:G^L\geq_\mathcal{U} H^L\,\,or\,\,\exists H^{LR}\in H^{L\mathcal{R}}\!:G\geq_\mathcal{U} H^{LR};$
  \item  $\forall G^R\in G^{\mathcal{R}}, \mbox{either }\exists H^R\in H^{\mathcal{R}}\!:G^R\geq_\mathcal{U} H^R\,\,or\,\,\exists G^{RL}\in G^{R\mathcal{L}}\!:G^{RL}\geq_\mathcal{U} H.$
\end{enumerate}
\end{theorem}

Note that the Common Normal Part implies the existence of an order-preserving map of mis\`ere-play into normal-play. That is, consider a universe of mis\`ere games, $\U\subseteq\M$. Then $$G\geq_\U H\Rightarrow G\geq_\text{normal-play}H.$$ 
To gain some intuition, we use {\sc Chess} terminology, for a large positive integer $n$, then the game $\left\{\bs{n}\mid-\bs{n}\right\}$ acts like a ``large zugzwang''; players do not want to be the first to play in this position. Therefore, in both $G+\left\{\bs{n}\mid-\bs{n}\right\}$ and $H+\left\{\bs{n}\mid-\bs{n}\right\}$, for a large $n$, players should play $G$ and $H$ with a ``normal-play strategy'', trying to get the last move and force the opponent to play first on the zugzwang part. Thus, if we don't have $G\geq_\text{normal-play}H$, we cannot have $G\geq_\mathcal{E}H$.

We now prove the main results of this subsection, which will allow us to establish inequality of games in $\mathcal{D}$ and $\mathcal{E}$ by studying only the options of the two games. The proof requires the use of Theorem~\ref{thm:commonnormalpart}, and so observe that both our universes are absolute (see also \cite{LarssonNPS} for a justification of this fact).

In the proof below, we use  inequalities such as  $o_L(G+X)\ge o_R(G^L+X)$, which follow immediately from the definition of left and right outcome: in this case, Left's outlook playing first in $G$ is at least as good as her outlook after she makes one particular move in $G$, at which point it is Right's turn.

\begin{theorem}[Subordinate Comparison in $\mathcal{D}$] \label{thm:comparisond}
Let $G,H\in\mathcal{D}$. Then $G\geq_\mathcal{D} H$ if and only if
\begin{enumerate}
\item $o(G)\geq o(H)$;
  \item$\forall H^L\in H^{\mathcal{L}}, \mbox{ either }
  \exists G^L\in G^{\mathcal{L}}\!:G^L\geq_\mathcal{D} H^L\,\,or\,\,\exists H^{LR}\in H^{L\mathcal{R}}\!:G\geq_\mathcal{D} H^{LR};$
  \item $\forall G^R\in G^{\mathcal{R}}, \mbox{ either }\exists H^R\in H^{\mathcal{R}}\!:G^R\geq_\mathcal{D} H^R\,\,or\,\,\exists G^{RL}\in G^{R\mathcal{L}}\!:G^{RL}\geq_\mathcal{D} H.$
\end{enumerate}
\end{theorem}

\begin{proof}
\noindent
($\Rightarrow$)
It is trivial that $G\geq_\mathcal{D} H$ implies item 1. Since $\D$ is absolute, $G\geq_\mathcal{D} H$ implies items 2 and 3 due to Theorem \ref{thm:commonnormalpart}.\\

($\Leftarrow$)
Assume $1$, $2$ and $3$, and also suppose that $G \not \geq_\mathcal{D} H$. By the definition of the partial
order, there is a distinguishing game $X$ such that either $o_L(G+X)<o_L(H+X)$ or $o_R(G+X)<o_R(H+X)$.
Choose $X$ to be of the smallest rank such that one or both hold, say $o_L(G+X)=\rm{R}<o_L(H+X)=\rm{L}$ ( 
if the simplest case is instead $o_R(G+X)<o_R(H+X)$, the argument is similar). There are three cases:

\begin{enumerate}[(a)]
 
 \item $H=0$ and $X=0$.
 
\noindent In this case, $o_L(H)=\rm{L}$ and $o_L(H+X)=\rm{L}$. Due to the assumption, $o_L(G)=o_L(G+X)=\rm{R}$. Therefore, $o_L(G)<o_L(H)$, contradicting $o(G)\geq o(H)$.

  \item $o_L(H+X)=o_R(H^L+X)=\rm{L}$, for some $H^L\in H^{\mathcal{L}}$.

  \noindent In this case, because of part $2$, we have either $G^L\geq_\mathcal{D} H^L$ or $G\geq_\mathcal{D} H^{LR}$. If the first holds,
then $o_L(G+X)\ge o_R(G^L+X)\ge o_R(H^L+X)=o_L(H+X)=\rm{L}$. If the second holds, then
 $o_L(G+X)\geq o_L(H^{LR}+X)\geq o_R(H^L+X)=o_L(H+X)=\rm{L}$. Both
 contradict the assumption $o_L(G+X)=\rm{R}<o_L(H+X)=\rm{L}$.

  \item $o_L(H+X)=o_R(H+X^L)=\rm{L}$, for some $X^L\in X^{\mathcal{L}}$.

\noindent By the ``smallest rank'' assumption, $o_R(G+X^L)\geq o_R(H+X^L)$. Therefore, $o_L(G+X)\geq o_R(G+X^L)\geq o_R(H+X^L)=o_L(H+X)=\rm{L}$. Once more, we contradict $o_L(G+X)=\rm{R}<o_L(H+X)=\rm{L}$.

\end{enumerate}

In each case we get a contradiction, and so we have shown that $G\geq_\mathcal{D} H$.
\end{proof}

We have a very similar result for comparison in $\mathcal{E}$; conditions 2 and 3 are identical, and only 1 must be made stronger. Above, for the dicot universe, we require only that the outcome of $G$ is at least the outcome of $H$. This is needed for the base case of $X=\bs 0$. For the dead-ending universe, where there are non-zero ends, the base case is if $X$ is an arbitrary Left-end, and so then we need to consider the outcomes of $G,H$ with any Left-end; that is, we need the {\em strong} left outcomes, $\hat{o}(G)$ and $\hat{o}(H)$. 

\begin{theorem}[Subordinate Comparison in $\mathcal{E}$] \label{thm:comparisone}
Let $G,H\in\mathcal{E}$. Then, $G\geq_\mathcal{E} H$ if and only if
\begin{enumerate}
\item $\hat{o}(G)\geq \hat{o}(H)$;
  \item For all $H^L\in H^{\mathcal{L}}, 
  \exists G^L\in G^{\mathcal{L}}\!:G^L\geq_\mathcal{E} H^L\,\,or\,\,\exists H^{LR}\in H^{L\mathcal{R}}\!:G\geq_\mathcal{E} H^{LR};$
  \item For all $G^R\in G^{\mathcal{R}}, \exists H^R\in H^{\mathcal{R}}\!:G^R\geq_\mathcal{E} H^R\,\,or\,\,\exists G^{RL}\in G^{R\mathcal{L}}\!:G^{RL}\geq_\mathcal{E} H.$
\end{enumerate}
\end{theorem}

\begin{proof}
\noindent ($\Rightarrow$)
Since $\E$ is absolute, then, due to Theorem \ref{thm:commonnormalpart}, $G\geq_\mathcal{E} H$ implies items 2 and 3. Also, $G\geq_\mathcal{E} H$ implies item 1 because, if not, say, without loss of generality, that $\hat{o}_L(G)=\rm{R}$ and $\hat{o}_L(H)=\rm{L}$. If so, \emph{for all} Left-ends $Y$, $o_L(H+Y)=\rm{L}$, and  there is a Left-end $X$ such that $o_L(G+X)=\rm{R}$. This means $o_L(G+X)<o_L(H+X)$, which contradicts $G\geq_\mathcal{E} H$.

\noindent ($\Leftarrow$)
Assume items $1$, $2$ and $3$, and also suppose that $G \not \geq_\mathcal{E} H$. By the definition of the partial
order, there is a distinguishing game $X$ such that either $o_L(G+X)<o_L(H+X)$ or $o_R(G+X)<o_R(H+X)$.
Choose $X$ to be of the smallest rank such that one or both of those hold. Without loss of generality, say $o_L(G+X)=\rm{R}<o_L(H+X)=\rm{L}$.  So Left wins $H+X$ moving first. We consider the cases (a) $H+X$ is a Left-end; (b) Left's good move in $H+X$ is in $H$; and (c) Left's good move in $H+X$ is in $X$. 

\begin{enumerate}[(a)]

 \item $H=\left\{\; \mid H^\mathcal{R}\right\}$ and $X=\left\{\; \mid X^\mathcal{R}\right\}$.
\noindent
In this case, $o_L(H+Y)=\rm{L}$ for any  Left-end $Y$. So, $\hat{o}_L(H)=\rm{L}$. On the other hand, due to the assumption, $o_L(G+X)=\rm{R}$ and so $\hat{o}_L(G)=\rm{R}$. Therefore, $\hat{o}_L(G)<\hat{o}_L(H)$, contradicting $\hat{o}(G)\geq \hat{o}(H)$.

  \item $o_L(H+X)=o_R(H^L+X)=\rm{L}$, for some $H^L\in H^{\rm{L}}$.

  \noindent In this case, because of part $2$, we have either $G^L\geq_\mathcal{E} H^L$ or $G\geq_\mathcal{E} H^{LR}$. If the first holds,
then $o_L(G+X)\ge o_R(G^L+X)\ge o_R(H^L+X)=o_L(H+X)=\rm{L}$. If the second holds, then
 $o_L(G+X)\geq o_L(H^{LR}+X)\geq o_R(H^L+X)=o_L(H+X)=\rm{L}$. Both
 contradict the assumption $o_L(G+X)=\rm{R}<o_L(H+X)=\rm{L}$.

  \item $o_L(H+X)=o_R(H+X^L)=\rm{L}$, for some $X^L\in X^{\mathcal{L}}$.

\noindent By the ``smallest rank'' assumption, $o_R(G+X^L)\geq o_R(H+X^L)$. Therefore, $o_L(G+X)\geq o_R(G+X^L)\geq o_R(H+X^L)=o_L(H+X)=\rm{L}$. Once more, we contradict $o_L(G+X)=\rm{R}<o_L(H+X)=\rm{L}$.

\end{enumerate}

\noindent
If the simplest case is instead $o_R(G+X)<o_R(H+X)$, the argument is similar.
Hence, we have shown that $G\geq_\mathcal{E} H$.
\end{proof}

\begin{example}
Consider $G=\left\{-\bs{1}\mid \bs{1}\right\}$. Then, $\hat{o}_L(G) = L$ and $\hat{o}_R(G) = R$. Therefore, $\hat{o}(G)=\mathscr{N}=\hat{o}(\bs 0)$. Note that $G^R = \bs{1}$ and $G^{RL} = \bs 0\geq_\mathcal{E} \bs 0$, so Theorem~\ref{thm:comparisone} gives $G\geq_\mathcal{E} \bs 0$. Symmetrically, $G\leq_\mathcal{E} \bs 0$. Hence $G\equiv_\mathcal{E} \bs 0$. 
\end{example}
We will encounter a generalization of this theme in Proposition~\ref{thm:atomice} (2). These tests for inequality modulo dicots or dead-ending games will be used to find canonical forms for $\mathcal{E}$ in Section \ref{sec:uniqueness}, and to demonstrate the conjugate invertibility property for both $\mathcal{D}$ and $\mathcal{E}$ in Section~\ref{sec:conjugate}.

\section{Reductions and defined canonical forms}\label{sec:reductions}
At the core of classical CGT we find theory of reduction, via domination and reversibility; the goal is to arrive at a `simplest form' to represent a full equivalence class of games (in mis\`ere play a.k.a. mis\`ere quotient). We demonstrate how this can be achieved in the mis\`ere universes of dicot and dead-ending games.
\subsection{Domination and reversibility}\label{sec:reverse}

In normal-play, removing all dominated options and bypassing all reversible options result in a  canonical  form. The same definitions of domination and reversibility also give  canonical  forms for general mis\`ere-play \cite{MesdaO2007, Siege}, but the usual reversibility does not work for arbitrary universes. The problem arises when reversing through ends. 

In normal-play, \textit{bypassing a reversible option $A$ through $B$} is to replace $A$ by the Left options of $B$, even if the replacement set $B^{\mathcal{L}}$ is empty. This results in a simpler game equivalent to the original. 
For reversibility in restricted mis\`ere play, when reversing through ends, replacing a reversible option with nothing  does not necessarily give an equivalent game  (see \cite{MilleyRenault2017} for details). Alternate simplifications were developed for the dicot universe $\mathcal{D}$ in \cite{DorbecRSS}. Our goal in this subsection is to do the same for the universe $\E$ of dead-ending games, which will lead to the development of  canonical  forms for $\mathcal{E}$ in  Section \ref{sec:uniqueness}.

We first note that removing dominated options does work modulo any universe, as demonstrated in \cite{DorbecRSS} for the dicot case. As in normal play, if there are at least two Left options and one dominates the other (under mis\`ere play modulo some universe), then we can remove the dominated option. 
The  concepts in this section are defined only from Left's perspective, and the corresponding Right concepts are defined analogously. 

\begin{theorem}[\cite{DorbecRSS} Domination]\label{thm:domination} 
Let $\U\subset \M$ be a universe, and let $G=\left\{\GL\mid \GR\right\} \in \U$. If  $A, B\in \GL$ with $A\leq_\U B$ then $G \equiv_\U \left\{\GL\setminus\left\{A\right\}\mid \GR\right\}$. 
\end{theorem}

Note that if $G$ is in the universe $\mathcal{D}$ of dicots (or $\mathcal{E}$ of dead-ending games), a reduced game $\left\{\GL\setminus\left\{A\right\}\mid \GR\right\}$, as in Theorem~\ref{thm:domination}, is still in  $\mathcal{D}$ ($\mathcal{E}$). To address the problem with reversible options in restricted universes, we adapt the following definitions from \cite{LarssonNPS}. Note that in that paper, ends are called ``atomic games'', and hence they use the term `atomic-reversible' instead of `end-reversible'.

\begin{definition}
Consider a universe $\U\in \M$, and let $G\in \U$. Suppose that there are followers $A\in \GL$ and $B\in A^\mathcal R$ with $B\leq_\U G$. Then the Left option $A$ is \textit{reversible through} its Right option $B$. If $B^\mathcal{L}$ is non-empty then $B^\mathcal{L}$ is a \textit{replacement set} for $A$, and in this case  $A$ is \textit{open-reversible}. If the \emph{reversing option}, $B$, is a Left-end, that is, if $B^\mathcal{L}=\emptyset$, then $A$ is \textit{end-reversible}.
 \end{definition}

We begin by showing that for any universe $\U$, replacement of open-reversible options works as indicated by this definition; when replacing $A$ by the set $\BL$, one does obtain an equivalent game (similar to the classical normal-play theory). This is Theorem \ref{thm:nonatomic} below, a straightforward generalization of the result for dicots \cite{DorbecRSS}. We then consider the more intricate case of end-reversible options, specifically in the universes $\mathcal{D}$ (Theorems \ref{thm:atomicd} and \ref{thm:substituted}, \cite{DorbecRSS}) and $\mathcal{E}$ (remainder of Section \ref{sec:reverse}). 

\begin{theorem}[Open reversibility]\label{thm:nonatomic}
Consider any absolute universe $\mathcal{U}\subseteq \cal M $, and let $G\in U$. If $A$ is an open-reversible Left option, reversible through $B$, with replacement set $\BL$, then
 $G\equiv_\mathcal{U} \left\{(\GL\setminus\left\{A\right\})\cup B^\mathcal{L} \mid {\GR}\right\}$.
\end{theorem}

\begin{proof}

Consider $G, A, B$ as in the statement of the theorem, and recall that, since
$B$ is a reversing Right option, then $G\geq_\mathcal{U} B$. Moreover, there is a non-empty replacement set
$B^\mathcal{L}$, so we let
$H = \left\{ \GL\setminus\left\{A\right\},B^\mathcal{L} \mid {\GR}\right\}$, and note that, by the parental property $H\in U$.
 We need to prove that $H\equiv_\mathcal{U} G$, i.e., $o(G+X)=o(H+X)$ for all $X\in U$. We proceed by induction on the rank of $X$.

 Fix $X$.
 Note that $B^\mathcal{L}$, $\GL$ and $\HL$ are non-empty so that $B+X$, $G+X$
 and $H+X$ all have Left options. Moreover $A+X$ has Right options.\\

\noindent \textit{Claim~1:}  $H\geq_\mathcal{U} B$.\\

\noindent \textit{Proof of Claim~1:}
Suppose that Left starts in the game $B+X$. If $C\in B^\mathcal{L}$ then $C\in\HL$ and thus, if Left wins in $B+X$ with $C+X$, Left also wins in $H+X$ with $C+X$. For the base case $\XL=\emptyset$, this is enough. If $\XL\ne \emptyset$, if Left wins in $B+X$ with $B+X^L$, by induction, Left also wins in $H+X$ with $H+X^L$.

Suppose that Right starts in the game $H+X$. If Right wins in $H+X$ with $H+X^R$ then, by induction, Right wins in $B+X$ with $B+X^R$. For the Right moves $H^R+X$ (this includes the base case $X=0$), we observe that $H^R+X$ is also a Right option of $G+X$. Since $G\geq_\mathcal{U} B$, if Right wins with $H^R+X$, we must have a winning move for Right in $B+X$.
 This concludes the proof of Claim~1.\\

To prove that $G\equiv_\mathcal{U} H$, we have to prove that $G\geq_\mathcal{U}H$ and $G\leq_\mathcal{U} H$. Besides induction, we only need to use two arguments.

For $G\geq_\mathcal{U}H$, if $B^L+X$ is a winning move for Left in $H+X$, then because $G\geq_\mathcal{U}B$, there is a winning move for Left in $G+X$.

For $G\leq_\mathcal{U} H$, if $A+X$ is a winning move for Left in $G+X$, after an automatic Right reply to $B+X$, Left must have a winning move again. But, by Claim~1, the existence of that winning move implies that Left has a winning move in $H+X$.

Thus $o(H+X)=o(G+X)$, for all $X$, and so $G\equiv_\mathcal{U} H$.
\end{proof}

This shows that reversibility through non-ends works just as you would expect: if the inequality for the reversibility is modulo $\U$, replace the reversible option with its replacement set and you obtain an  equivalent game modulo $\U$. Before we can deal with end-reversibility --- where there is no replacement set --- we must prove a strategic fact for Left playing in a game $G$ with an end-reversible option.

\begin{lemma}[Weak Avoidance Property]\label{lem:WeakAvoidanceProperty}
Let $G\in \mathcal{U}$ and let $A$ be an end-reversible Left option of $G$.
Let $X$ be a game such that $\XL\ne \emptyset$. If Left wins $G+X$ with $A+X$, then Left also wins $G+X$ with some  $G+X^L$.
\end{lemma}

\begin{proof}
Let $A$ be an end-reversible Left option of $G$ and let $B\in A^\mathcal{R}$ be
a reversing option for $A$. 

 By assumption, $G\geq_\mathcal{U} B$ and $B^{\cal L}=\emptyset$. 
 Since $B\in A^{\cal R}$, then $B+X$ is a Right option of $A+X$. Consequently, if Right, playing first, loses $A+X$, then Left, playing first, wins $B+X$. Of course, because $B^{\cal L}=\emptyset$, the Left winning move must be some $B+X^L$. But $G\geq_\mathcal{U} B$, so $G+X^L$ must be a winning move for Left, from $G+X$.\end{proof}

The remainder of this subsection deals with the interesting case of finding a replacement set for an end-reversible option. 
There are essentially two cases. If Left has a good move other than the reversible option, then it can be eliminated, and otherwise, it can only be simplified. The specific simplification depends on the universe. The simplification for dicot mis\`ere is proven in \cite{DorbecRSS}. Note that in dicot mis\`ere, the only Left-end is the game \bs 0, so the following result covers all Left-end cases in dicot play.

    \begin{theorem}[\cite{DorbecRSS} End-reversibility in $\mathcal{D}$]\label{thm:atomicd}
Let $G\in \mathcal{D}$ and suppose that $A\in\GL$ is reversible through \bs 0. Then, either
\begin{enumerate}
\item Left has a winning move in $ G^{\cal L}\setminus\left\{A\right\}$, and 
$G\equiv_\mathcal{D} \left\{ \GL\setminus \left\{A\right\}\mid {\GR}\right\}$, or 
\item Left's option $A$ is the only winning move in $\GL$, and $G\equiv_\mathcal{D} \left\{ *,\GL\setminus\left\{A\right\}\mid {\GR}\right\}$.
\end{enumerate}
\end{theorem}

Note that here $A$ is reversible through \bs 0, so $G\geq_\mathcal{D} \bs 0$, which means that Left wins $G$ playing first (for example to the redundant option $A$ !). 

In the special case where $G$ has only one left and one Right option, both end-reversible, then the following substitution can be made for $G$. Here we are applying Theorem \ref{thm:atomicd} to both end-reversible options to get $\left\{*\mid *\right\}$, and then using the known result that $\left\{* \mid *\right\}$, or $*+*$, is equivalent to \bs 0 modulo dicots \cite{Allen2011, MckayMN}.

\begin{theorem}[\cite{DorbecRSS} Substitution Theorem for $\mathcal{D}$]\label{thm:substituted}
If $G=\left\{A\mid C\right\}$ where $A$ and $C$ are end-reversible options
 then $G\equiv_\mathcal{D} \bs 0$.
\end{theorem}

We will now extend the technique of \cite{DorbecRSS} to develop a similar solution for end-reversibility in $\mathcal{E}$, and we introduce the concept of a \emph{fundamental option}, which is a singular survival option for a given player, since as we will see; it ensures victory even following an attack by the worst perfect murder component.

\begin{definition}
Let $G\in\mathcal{E}$. Then $A\in \GL$ is a \emph{fundamental Left option} if $\hat{o}_L(G) = \rm L$ and $\hat{o}_L\left (\left\{ \GL\setminus \left\{A\right\}\mid {\GR}\right\}\right ) = \rm R$.
\end{definition}

That is, if $A$ is a fundamental Left option of a game $G$, then there is a Left-end $E$ such that $A+E$ is the one and only good Left move in $G+E$. 

(Similarly, if $C$ is a fundamental Right option then there is a Right end $Y$ such that $C+Y$ is the one and only good right move in $G+Y$.)

We are now ready for our first end-reversible results for $\mathcal{E}$.  There are more cases to consider than there were in  $\mathcal{D}$. We start by proving that if an end-reversible option is not  fundamental, then it can be removed to leave an equivalent game.  It is extremely important to note that we will only remove a lone Left option if the result is still in $\mathcal{E}$! If the Right options of $G$ are not all Left-ends, then changing $G=\left\{A \mid \GR \right\}$ to $\left\{\; \mid \GR \right\}$ would create a game that is not dead-ending. In this case, we will instead settle for simplifying $G$ using a substitution (Theorem \ref{thm:substitutee}). 

In the proof below, we must consider the two cases  $|\GL|=1$ and $|\GL|>1$.
If $|\GL|=1$ then an end-reversible Left option cannot be fundamental, because its removal leaves a Left-end. Thus, as long as we  stay inside the universe, we can always remove an end-reversible option if it is the only option. If there is also only one Right option, and it too is end-reversible, then we can remove the options simultaneously and obtain $G\equiv_\mathcal{E}\bs 0$. 

\begin{theorem}[End-reversibility in $\mathcal{E}$]\label{thm:atomice}
Let $G\in \mathcal{E}$ and suppose that there is an end-reversible option $A\in\GL$.
\begin{enumerate}
\item[(1)] If $A$ is not a fundamental Left option of $G$ and $\left\{\GL\setminus \left\{A\right\} \mid \GR\right\} \in \mathcal{E}$,
then
$$G\equiv_\mathcal{E}\left\{ \GL\setminus \left\{A\right\}\mid {\GR}\right\}.$$
\item[(2)] If $G=\left\{A\mid C\right\}$, where $C$ is also end-reversible, then we remove both options and get $G\equiv_\mathcal{E} \bs 0$.
\end{enumerate}

\end{theorem}

\begin{proof} (1) Consider  $A\in \GL$, $B\in A^\mathcal R$  as in the
statement of the theorem, with $B=\left\{\;\mid  \BR \right\}$ and $G\geq_\mathcal{E} B$.
Note that,  since $B$ is a Left-end, its strong left outcome is $\rm{L}$, so  $G\geq_\mathcal{E} B$ implies that $\hat{o}_L(G)=\rm{L}$.

Let  $H=\left\{ \GL\setminus \left\{A\right\} \mid  G^\mathcal{R}\right\}$; by assumption, $H\in \mathcal{E}$ (note that $H\not \in \mathcal{E}$ is only possible, but not guaranteed, if $|\GL|=1$). Let $X\in \mathcal{E}$. We prove by induction on the rank of $X$ that $o(H+X)= o(G+X)$. 

Suppose Left wins $H+X$. If $|\GL|=1$, then Left has no move in $H$, so the winning move must be to $H+X^L$. By induction, $G+X^L$ is also a winning move, where in the base case of $X^L=\emptyset$, Left wins $G+X$ because the strong left outcome of $G$ is $\bs{L}$. If  $|\GL|>1$, then $G\geq_\mathcal{E} H$ by the hand-tying principle (Lemma \ref{lem:greediness}). So either way we have 
$o(G+X)\geq o(H+X)$. 

Next, suppose Left wins $G+X$.
If Left's winning move in $G+X$ is $G'+X$ where $G'\in \GL\setminus\left\{A\right\}$, then this is also a left winning move in $H+X$.
If Left's winning move in $G+X$ is $G+X^L$, then by induction $o(H+X^L)\geq o(G+X^L)$ and so $H+X^L$ is a left winning move in $H+X$.
The remaining case is if $A+X$ is the only winning move in $G+X$. In this case, $X$ must be a Left-end; else by Lemma~\ref{lem:WeakAvoidanceProperty}, there would exist a winning move $G+X^L$, a contradiction. Now, if $X$ is a Left-end, 
then $|\GL|=1$ means $H$ is also a Left-end, so Left wins $H+X$. If $|\GL|>1$,
since $B+X$ is a Left-end we know $\hat{o}_L(B)=\rm{L}$, and then since $G\geq_\mathcal{E} B$, we have $\hat{o}_L(G) =\rm{L}$. But then we must have $\hat{o}_L(H) =\rm{L}$; else $A$ would be a fundamental Left option. So Left wins $H+X$. 
Thus, we have $o(H+X)\geq o(G+X)$, and so $o(H+X)=o(G+X)$ for all $X\in \mathcal{E}$, which means $G\equiv_\mathcal{E} H$. 
Notice that we have this equivalence even if $H\not \in \mathcal{E}$ (our arguments do not require it), but we will only apply the simplification if $H$ stays in $\mathcal{E}$.

(2) If $G=\left\{A\mid C\right\}$, both options end-reversible, then the exact argument above shows $G \equiv_\mathcal{E} \left\{\; \mid C\right\} \equiv_\mathcal{E} \left\{\; \mid \; \right\}$, where in this case of `simultaneous' removal, we do not care if the game  $\left\{\; \mid C\right\}$ is not in  $\mathcal{E}$.
\end{proof}

Next we have a general structural simplification for any end-reversible Left option $A$; in particular this allows us to simplify end-reversible options that do not satisfy Theorem \ref{thm:atomice}, either because they are fundamental or because they are lone Left options whose removal bumps the game out of $\mathcal{E}$.

\begin{lemma} \label{lem:atomice}
Let $G\in \mathcal{E}$. If $A\in\GL$ is reversible through a Left-end, say $E=\left\{\; \mid E^\R \right\}$, then 
$G\equiv_\mathcal{E}  \left\{\left\{\; \mid  E\right\},\GL\setminus\left\{A\right\}\mid \GR \right\}$.
\end{lemma}

\begin{proof}
 Let $H=\left\{ \left\{\; \mid  E\right\},\GL\setminus\left\{A\right\}\mid {\GR}\right\}$, and consider any $X\in \mathcal{E}$. Note that $H\in \E$, since $G,E\in \E$, and suppose that $o_L(H+X)=\rm L$.
If $\left\{\; \mid  E\right\}+X$ is a winning move for Left in $H+X$, then Left has a good response to Right's move to $E+X$, and this means that Left can win also in the game $G+X$, because, by assumption,  $G\geq_\mathcal{E} E$. Any other winning option for Left in $H+X$ is also available in $G+X$, so $o_L(G+X)=\rm L$. Observe that any winning option for Right in $G+X$ is also available in $H+X$, so trivially $o_R(H+X)=o_R(G+X)$. Altogether, this shows that $G\geq_\mathcal{E} H$. 

We will prove by induction on the rank of $X$ that we also have $o(H+X)\geq o(G+X)$.

Suppose Left wins $G+X$, playing first. 
If Left's winning move in $G+X$ is $G+X^L$, then by induction $o(H+X^L)\geq o(G+X^L)$ and so $H+X^L$ is a left winning move in $H+X$.
If Left's winning move in $G+X$ is $G'+X$ where $G'\in \GL\setminus\left\{A\right\}$, then this is also a left winning move in $H+X$.
The remaining case is if $A+X$ is the only winning move in $G+X$. In this case, $X$ must be a Left-end; else by Lemma~\ref{lem:WeakAvoidanceProperty}, there would exist a winning move $G+X^L$, a contradiction. But if $X$ is a Left-end, then $\left\{\; \mid  E\right\}+X$ is a winning move for Left, because in $\mathcal{E}$ Left-ends remain Left-ends. Thus Left wins $H+X$.
\end{proof}

In fact, we can do better than  Lemma \ref{lem:atomice}, and simplify instead to a game of the same form but with $B$ replaced by a position with smaller (or at least not-larger) rank. This is shown next in the Substitution Theorem (Theorem \ref{thm:substitutee}). To get there, we use the perfect murder games --- specifically the fact that they are worse than any other Left-end --- to construct another substitution for end-reversible options. Like Lemma \ref{lem:atomice}, this substitution applies to all end-reversible options, but is only useful for those options that are not removed by Theorem \ref{thm:atomice}: that is, when $A$ is fundamental (and therefore $|\GL|>1$), or when $A$ is a lone option and $\left\{\; \mid \GR\right\} \not \in \mathcal{E}$. 

\begin{theorem}[Substitution Theorem for $\mathcal{E}$]\label{thm:substitutee}
Let $G\in \mathcal{E}$. If   $A\in\GL$ is end-reversible, then  there exists a smallest nonnegative integer $n$ such that $G\geq_\mathcal{E} M_n$ and $G \equiv_\mathcal{E} \left\{\left\{\; \mid M_n\right\}, \GL\setminus\left\{A\right\} \mid \GR \right\} $.
\end{theorem}

\begin{proof}
Suppose that  $A\in\GL$ is end-reversible through $E=\left\{\; \mid  E^\mathcal{R}\right\}$. Let $k=\rank(E)$.
By assumption, $G\geq_\mathcal{E} E$ and, by Theorem~\ref{thm:pm}, $E\geq_\mathcal{E} M_k$, and thus $G\geq_\mathcal{E}M_k$. 
Since $k$ is a nonnegative integer, the existence part is clear. Let $n$ be the minimum nonnegative integer such that $G\geq_\mathcal{E} M_n$.

Let $H= \left\{\left\{\; \mid M_n\right\}, \GL\setminus\left\{A\right\} \mid \GR \right\} $, and let $G' = \left\{\left\{\; \mid M_n\right\}, \GL \mid \GR \right\} $. By Lemma~\ref{lem:greediness}, we have $G'\geq_\mathcal{E} G$, and so then $G \geq_\mathcal{E} M_n$ and  $G \geq_\mathcal{E} E$ imply that $\left\{\;\mid M_n\right\}$ and $A$ are end-reversible options of $G'$. 

Now, $\hat{o}_L(H)=\rm{L}$, because Left moving first in any $H+X$, for $X$ a Left-end, wins by moving to the Left-end $\left\{\; \mid M_n\right\}+X$. 
Likewise, $\hat{o}_L(G)=\rm{L}$ because $G\geq_\mathcal{E} E$ and $\hat{o}_L(E)=\rm{L}$.
Since $H$ and $G$ are the games $G'$ with $A$ removed and $G'$ with $\left\{\; \mid M_n\right\}$ removed, respectively, this shows that neither $A$ nor $\left\{\; \mid M_n\right\}$ is fundamental in $G'$. Thus by Theorem \ref{thm:atomice}, we get $G'\equiv_\mathcal{E} H$ and $G'\equiv_\mathcal{E} G$. This gives $G\equiv_\mathcal{E} H$, as required.
\end{proof}

Note that the rank of $\left\{ \; \mid E\right\}$ is greater than or equal to the rank of $\left\{ \; \mid M_n\right\}$, because $n$ is chosen to be minimum. So Theorem \ref{thm:substitutee} will usually give a substitution that is  simpler than the substitution in Lemma \ref{lem:atomice}. After making the substitution with $\left\{\; \mid M_n\right\}$, we observe that there is only one position in $\mathcal{E}$ of the particular form $\left\{A\mid C\right\}$ from Theorem \ref{thm:atomice}.

It is interesting to note that, after applying the Substitution Theorem above, there is actually only one game of the form described in Theorem \ref{thm:atomice} (2).
If $G=\left\{A\mid C\right\}\in \mathcal{E}$ with both $A$ and $C$ end-reversible, then by Theorem \ref{thm:substitutee}, 
$$G\equiv_\mathcal{E}\left\{ \left\{\; \mid M_n\right\} \mid \left\{-M_k\mid \;\right\} \right\}.$$ Since the  options are reversible, we have $-M_k\geq_\mathcal{E} G\geq_\mathcal{E} M_n$; but $-M_k\geq_\mathcal{E} M_n$ is only possible if both are zero, because otherwise the former is Right-win and the latter is Left-win. So $k=n=0$ and $G=\left\{ \left\{\; \mid \bs 0\right\} \mid \left\{\bs 0\mid \;\right\} \right\} = \left\{-\bs 1\mid \bs 1\right\}$. (This game would then reduce to \bs 0 by Theorem \ref{thm:atomice}.)

These are all of the reductions for the dead-ending universe.
To summarize, there are four types of reductions in  $\mathcal{D}$ \cite{DorbecRSS} and five types in  $\mathcal{E}$ (new). The first two apply in any $\U$, including $\mathcal{D}$ and $\mathcal{E}$:
\begin{enumerate}
  \item Remove dominated options. (Theorem \ref{thm:domination})
  \item Reverse open-reversible options. (Theorem \ref{thm:nonatomic})
\end{enumerate}
In $\mathcal{D}$ we have three additional reductions, for end-reversible options \cite{DorbecRSS}:
\begin{enumerate}
\item Remove end-reversible options, as long as there is another winning option
  \item[4.] Replace end-reversible options by $*$. (Theorem \ref{thm:atomicd})
  \item[5.]  Replace $\left\{* \mid *\right\}$ by \bs 0. (Theorem \ref{thm:substituted})
\end{enumerate}
And in $\mathcal{E}$ we have three additional reductions, for end-reversible options:
\begin{enumerate}
  \item[3.] Remove non-fundamental end-reversible options, including removal of lone options as long as the result is in $\mathcal{E}$. (Theorem~\ref{thm:atomice} (1))
  \item[4.] Simultaneously remove lone left and lone Right end-reversible options; i.e., replace $\left\{A\mid C\right\}$ with \bs 0 if both $A$ and $C$ are end-reversible. (Theorem~\ref{thm:atomice} (2))
  \item[5.] Replace other end-reversible options (those that are fundamental, or those whose removal gives a game not in $\mathcal{E}$) by $\left\{\; \mid M_n\right\}$ for Left options, or by $\left\{-M_n \mid \; \right\}$ for right. (Theorem~\ref{thm:substitutee})
\end{enumerate}

We will call a game {\em reduced} if the above reductions have all been applied, as described below, and in Section \ref{sec:uniqueness} we will show that the simplest reduced form of a game in $\mathcal{E}$ is unique. 

\begin{definition}\label{def:reducedform}
A game  $G\in \mathcal{E}$ is said to be \textit{reduced in $\mathcal{E}$} if none of Theorems~\ref{thm:domination},
\ref{thm:nonatomic},~\ref{thm:atomice}, or \ref{thm:substitutee} can be applied to
$G$ to obtain an equivalent game in $\mathcal{E}$ with different sets of options. 
A game  $G\in \mathcal{D}$ is said to be \textit{reduced in $\mathcal{D}$} if none of Theorems~\ref{thm:domination},
\ref{thm:nonatomic},~\ref{thm:atomicd}, or~\ref{thm:substituted} can be applied to
$G$ to obtain an equivalent game in $\D$ with different sets of options.
\end{definition}

\subsection{Uniqueness and simplicity of reduced forms}\label{sec:uniqueness}

We are now able to prove the existence of a unique and simplest reduced form for a congruence class of games in the dead-ending universe. This was shown for the dicot universe in \cite{DorbecRSS}. We will use  $\cong$ to indicate that two games are ``identical''; that is, $G\cong H$ if the positions $G$ and $H$ have identical game tree structure.

First, we note the following result from \cite{MilleyRenault}, which says that all ends are invertible modulo $\mathcal{E}$.

\begin{theorem}\label{thm:endinverses} {\em \cite{MilleyRenault}}
If $G\in \mathcal{E}$ is an  end, then $G$ is invertible and the inverse is the conjugate, i.e. 
$$G+\Conj{G}\equiv_\mathcal{E}\bs 0.$$
\end{theorem}

The main result is Theorem \ref{thm:uniqred} below. Recall that a game in $\mathcal{E}$ is {\em reduced} if all dominated options have been removed; all open-reversible options have been reversed; all non-fundamental end-reversible options have been removed, unless removal gives a non-dead-ending game; and all other end-reversible options have been replaced by $\left\{\; \mid M(n)\right\}$ (left) or $\left\{-M(n) \mid \; \right\}$ (right). We now show that any two equivalent reduced games are in fact identical.

\begin{theorem}\label{thm:uniqred}
Let $G,H\in \mathcal{E}$. If $G\equiv_\mathcal{E} H$ and both are reduced games, then $G\cong H$.
\end{theorem}

\begin{proof}

We will proceed by induction on the sum of the ranks of $G$ and $H$.
We will exhibit a correspondence $G^{L_i}\equiv_\mathcal{E} H^{L_i}$ and $G^{R_j}\equiv_\mathcal{E} H^{R_j}$
between the options of $G$ and $H$.
By induction, it will follow that $G^{L_i}\cong H^{L_i}$, for all $i$,  and $G^{R_j}\cong H^{R_j}$, for all $j$,
and consequently  $G\cong H$.

For the base case, if $G =\bs 0$ and $H =\bs 0$, then  $G\cong H$.
If $G$ and $H$ are not both zero, then without loss of generality,  assume that there is a Left option $H^L$. We will break the proof into two cases based on whether or not $H^L$ is an end-reversible option.

\begin{description}
\item {\em Case 1:} $H^L$ is not end-reversible.

Note that this means $H^L$ is not reversible at all, since $H$  cannot have open-reversible options if it is reduced.
Since $G \equiv_\mathcal{E} H$, then of course $G\geq_\mathcal{E} H$. Then from Theorem~\ref{thm:comparisone},
there exists a $G^L$ with $G^L\geq_\mathcal{E} H^L$ or there exists a $H^{LR}\leq_\mathcal{E} G$.
Now $H^{LR}\leq_\mathcal{E} G\equiv_\mathcal{E} H$ would contradict that $H^L$ is not reversible.
Thus, there is some $G^L$ with $G^L\geq_\mathcal{E} H^L$.

We claim that this $G^L$ cannot be end-reversible. If it was, then since $G$ is in reduced form, 
$G^L \cong \left\{ \;  \mid M(n)\right\}$
for some nonnegative integer $n$, and  $G\geq_\mathcal{E}M(n)$.
Since $G\equiv_\mathcal{E} H$ we also have $H\geq_\mathcal{E} M(n)$.
Therefore, using the invertibility of  $\left\{ \;  \mid M(n)\right\}$ (Theorem \ref{thm:endinverses}), we have
\[G^L\geq_\mathcal{E} H^L\Leftrightarrow \left\{ \;  \mid M(n)\right\}\geq_\mathcal{E} H^L \Leftrightarrow 0\geq_\mathcal{E}  H^L+\left\{\Conj{M(n)} \mid \;\right\},\]
where the last equivalence is by Lemma~\ref{lem:dis}.
From Theorem~\ref{thm:comparisone} (2), the last inequality implies that
from every left move in  $H^L+\left\{\Conj{M(n)}\mid \;\right\}$, there exists a right move that is less than $0$. There is no right response in $\left\{\Conj{M(n)} \mid \;\right\}$, so this right move must look like $H^{LR}+\Conj{M(n)}\leq_\mathcal{E}0$.
The inequalities
$H\geq_\mathcal{E} M(n)$ and $H^{LR}+\Conj{M(n)}\leq_\mathcal{E}0$ give that $H\geq_\mathcal{E} H^{LR}$, which contradicts that $H^L$ is not reversible. Thus, $G^L$ is not end-reversible, as claimed.

A similar argument for $G^L$ gives a Left option $H^{L'}$ such that $H^{L'}\geq_\mathcal{E} G^L$.
Therefore, $H^{L'}\!\geq_\mathcal{E} G^L\!\geq_\mathcal{E} H^L$. Since there are no dominated options in the reduced game $H$, it must be that
$H^{L'}\!\equiv_\mathcal{E} H^L\!\equiv_\mathcal{E} G^L$. By induction, $H^L\cong G^L$.

The symmetric argument gives that each non-reversible option $H^R$ is identical to some
$G^R$. In conclusion, we have a \emph{pairwise correspondence} between
 options of $G$ and $H$ that are not end-reversible.\\

\item {\em Case 2:} $H^L$ is end-reversible, so $H^L\cong \left\{ \;  \mid M(n)\right\}$ and $H\geq_\mathcal{E} M(n)$. We have two subcases.
\\

{\em Case 2a:} $|\HL|>1$.

In this case, since $H$ is reduced and Theorem~\ref{thm:atomice}  cannot be further applied,
it must be that $H^L$ is a fundamental Left option of $H$.

This means there is a Left-end $X$ such that $H^L+X$ is the only winning move for Left in $H+X$. Since $G\equiv_\mathcal{E} H$, there must also be a good left move in $G+X$, say $G^L+X$. We claim that $G^L$ is end-reversible. If not, then by Case 1, there is a corresponding non-reversible option $H^{L'}$ in $H$ such that $G^L\cong H^{L'}$. Then $H^{L'}+X$ is also a good move in $H+X$, a contradiction.

So, $G^L$ is end-reversible in $G$, and since $G$ is reduced this means $G^L=\left\{ \;  \mid M(n')\right\}$. To see $G^L\equiv_\mathcal{E} H^L \cong \left\{ \;  \mid M(n)\right\}$, we need only show $n=n'$: but this follows from Theorem~\ref{thm:substitutee}, because $n, n'$ are minimal such that $G\geq_\mathcal{E} M(n),M(n')$. Thus, $H^L\cong G^L$.
\\

 {\em Case 2b:} $|\HL|=1$. Recall $H^L = \left\{\; \mid M(n)\right\}$.

We need to show that $G$ has a Left option $G^L\cong H^L$; we claim that it suffices to show  $\GL\not =\emptyset$. To see that this is sufficient, first note that any Left options of $G$ must be end-reversible, since otherwise the pairwise correspondence from Case 1 would mean $H$ has a non-reversible Left option, contradicting our assumptions for Case 2b (there is only one Left option of $H$, and it is end-reversible).
Since $G$ is reduced, it has at most one end-reversible Left option, as the rest would be removed by Theorem \ref{thm:atomice}, and so if $\GL\not =\emptyset$ then in fact $|\GL|=1$.  By Theorem \ref{thm:substitutee}, this one end-reversible Left option must be $G^L=\left\{\; \mid M(n')\right\}$, and then $n=n'$ follows as above, and we get $G^L\cong H^L$, as required.

So suppose by way of contradiction that $\GL=\emptyset$. 
Since $H=\left\{ \left\{ \;  \mid M(n)\right\}\mid \HR\right\}$ cannot be reduced further, 
Theorem \ref{thm:atomice} (1) cannot be applied, and it follows that $\left\{ \; \mid \HR\right\}\not\in \mathcal{E}$. Thus, there must exist some Right option $H^R$ of $H$ that is  not a Left-end. If this $H^R$ is not end-reversible then by Case 1 there is a corresponding Right option $G^R\cong H^R$ in $G$; but this is a contradiction because $G\in \mathcal{E}$ would be a Left-end with a follower that is not a Left-end. So $H^R$ is end-reversible. If $H^R$ is the only Right option then $H$ is of the form in Theorem \ref{thm:atomice} (2), which contradicts the fact that $H$ is reduced. So there are other Right options, say $H^{R_2},H^{R_3},\ldots$, and then we know two things: first, $H^R$ must be fundamental, else it would have been removed, and second, these other Right options  cannot be end-reversible, else as non-fundamental,  {\em they} would have been removed. By case 1, this means there are corresponding non-reversible options in $\GR$,    $G^{R_2}\cong H^{R_2}, G^{R_3} \cong H^{R_3},\ldots$. Since $H^R$ is fundamental, there is a Right end $Y$ such that $H^R+Y$ is the only good right move in $H+Y$. In particular, none of $H^{R_2}+Y,H^{R_3}+Y,\ldots$ is a good move in $H+Y$. Since $G\equiv_\mathcal{E} H$, there must be a winning move in $G+Y$, say $G^R+Y$. This $G^R$ must be end-reversible, else by Case 1 it would be identical to one of the $H^R_i$. But by Theorem \ref{thm:substitutee}, since $G$ is reduced, $G^R$ must be $\left\{\Conj{M(m)}\mid \;\right\}$, and this is a contradiction because all options of the Left-end $G\in \mathcal{E}$ must be Left-ends.

This shows that $\GL=\emptyset$ is impossible, and so then by the argument above we get $\GL=\left\{G^L\right\}$ with $G^L\cong H^L$.
\end{description}
In all cases, we have shown that $\HL$ is identical to $\GL$. The proof for $\HR$ and $\GR$ is similar. Consequently, $G\cong H$.
\end{proof}

We have shown that two equivalent reduced games must have identical game trees; this gives the following immediate corollary. 

\begin{corollary}[Uniqueness]
If $G\in \mathcal{E}$ then there is a unique reduced form of~$G$ modulo $\mathcal{E}$.
\end{corollary}

We can further say that the reduced form of a game in $\mathcal{E}$ is the simplest equivalent form, modulo $\mathcal{E}$.

\begin{theorem}[Simplicity]
Let $G\in \mathcal{E}$ be a reduced form of a game. If $G'\equiv_\mathcal{E}G$, then $\rank(G')\geq \rank(G)$.
\end{theorem}

\begin{proof}
If $G'$ is also reduced, then $G' \cong G$ and so clearly the ranks are equal. Otherwise, reduce $G'$ to $G''$; 
then $G''\equiv_\mathcal{E} G' \equiv_\mathcal{E} G$ implies $G''\cong G$ by Theorem \ref{thm:uniqred}. Since all reductions either maintain or reduce the rank of $G''$, this means $\rank(G')\geq \rank(G'')=\rank(G)$.
\end{proof}

These results allow us to talk about
\textit{the canonical form} of a game in the dead-ending universe, $\mathcal{E}$. Analogous results were shown for $\mathcal{D}$ in \cite{DorbecRSS}. With this, we have completed our first major goal of the paper. In the next and final section, we use  canonical  forms to prove that both $\mathcal{D}$ and $\mathcal{E}$ have no non-conjugal inverses.

\section{Conjugates and inverses}\label{sec:conjugate}

It is possible in general for a position to have an inverse that is not its conjugate. We call this a {\em non-conjugal} inverse, and we are very interested in knowing when this is or is not possible; certainly it seems that well-structured universes should not allow such a thing. We end the paper with a proof that no non-conjugal inverses are possible modulo $\mathcal{D}$ or $\mathcal{E}$. We will find good use of the defined canonical forms from the previous section.

\begin{definition} A game $G$ in a universe $\mathcal{\mathcal{U}}$ has a {\em non-conjugal inverse} if there exists an $H\in \mathcal{U}$  with  $H\not \equiv_\mathcal{U} \Conj{G}$ such that $G+H\equiv_\mathcal{U} 0$.
\end{definition}

\begin{definition} A universe $\mathcal{U}$ has  \textit{the conjugate property} if no game in $\U$ has a non-conjugal inverse; that is, 
if $G+H\equiv_\mathcal{U} 0$ implies $H\equiv_\U \Conj{G}$, for all $G,H\in \mathcal{U}$. 
\end{definition}

And now the theorems: first a full proof for $\mathcal{E}$ and then a mostly analogous proof for $\mathcal{D}$. The proof (specifically of claim(ii) in Case 1) uses a technique originally employed by Ettinger in his work on dicot scoring games  \cite{Ettinger}.

\begin{theorem}\label{Conjugate} The universe $\mathcal{E}$ of dead-ending games has the conjugate property; that is, for all 
$G,H\in \mathcal{E}$, if $G+H\equiv_\mathcal{E} 0$ then $H\equiv_\mathcal{E} \Conj{G}$.
\end{theorem}

\begin{proof}
Let $G,H\in \mathcal{E}$ satisfy $G+H\equiv_\mathcal{E} 0$, and consider $G,H$ in their (unique) reduced, canonical forms.
We will prove, by induction on the rank of $G+H$, that we must have $H\equiv_\mathcal{E} \Conj{G}$. 
In the below proof, because of numerous algebraic manipulations, we will revert to the short hand notation
$-G = \Conj{G}$, if the existence of a negative is given by induction.

For the base case, if $G$ and $H$ both reduce to $0$, then trivially $H\equiv_\mathcal{E} \Conj{G}$.
Otherwise, the game $G+H$ has at least one option; without loss of generality, assume $G$ has at least one Left option, $G^L$.
 We break the proof into two cases based on whether or not $G^L$ is end-reversible.
\\

\noindent {\bf Case 1:}  $G^L$ is not end-reversible.

Note then that $G^L$ is not reversible at all, because only end-reversible options may exist in the reduced form. We prove two claims.\\

\noindent \textit{Claim (i):} There exists $H^{R}$ such that $G^{L}+H^{R}\leq_\mathcal{E} 0$.

Let $J=G+H \equiv_\mathcal{E} 0$. 
It is a consequence of Theorem~\ref{thm:comparisone} that for all Left moves $J^L$,
 there exists $J^{LR}$ such that $J^{LR}\leq_\mathcal{E}  0$.
In particular, for the Left option $G^L+H$, there exists either $G^{LR}+H\leq_\mathcal{E}  0$ or  $G^{L}+H^R\leq_\mathcal{E}  0$.
The former inequality is impossible, since adding $G$ to both sides
  (Lemma~\ref{lem:dis}) gives $G^{LR}\leq_\mathcal{E} G$, which contradicts that
 $G^L$ is not a reversible option. 
  Therefore, the claim holds.
  \\

\noindent \textit{Claim (ii):} With $H^R$ as in (i), $G^L+H^R\equiv_\mathcal{E} 0$.

  We have that $G^{L}+H^{R}\leq_\mathcal{E} 0$, and so suppose by way of contradiction that the inequality is strict.
Let us index the options starting with $G^L=G^{L_1}$ and $H^R=H^{R_1}$, so that our assumption is $G^{L_1}+H^{R_1}<_\mathcal{E} 0$.

 Consider the Right move in $G+H$ to $G+H^{R_1}$. Since $G+H\equiv_\mathcal{E} 0$, in particular $G+H\geq_\mathcal{E} 0$, and so there exists a Left option such that $(G+H^{R_1})^L\geq_\mathcal{E} 0$. This option is either of the form $G+H^{{R_1}L}\geq_\mathcal{E} 0$ or $G^{L_2}+H^{R_1}\geq_\mathcal{E} 0$ for some option $G^{L_2}$ that cannot be $G^{L_1}$, since 
 $G^{L_1}+H^{R_1}<_\mathcal{E} 0$.

 If we have the first inequality, $G+H^{{R_1}L}\geq_\mathcal{E} 0$, then, by adding $H$ to both sides,
 we get $H^{{R_1}L}\geq_\mathcal{E} H$. Therefore, since $H$ is in canonical form,
 $H^{R_1}$ is an end-reversible option and, by Theorem~\ref{thm:substitutee},
 $H^{R_1} = \left\{-{M(n)}\mid \; \right\}$  for some nonnegative integer $n$.
 The two inequalities $G^{L_1}+H^{R_1}\leq_\mathcal{E} 0$ and $G+H^{{R_1}L}\geq_\mathcal{E} 0$
 become $G^{L_1}+\left\{-{M(n)}\mid \; \right\}\leq_\mathcal{E} 0$ and $G-{M(n)}\geq_\mathcal{E} 0$, respectively.
 In $G^{L_1}+\left\{-{M(n)}\mid \; \right\}$, by Theorem~\ref{thm:comparisone}, if Left moves to $G^{L_1}-{M(n)}$,
 then Right has a response of the form $G^{{L_1}R}-{M(n)}\leq_\mathcal{E} 0$. Then 
 $G^{{L_1}R}\leq_\mathcal{E} {M(n)}$ and
 $G \geq_\mathcal{E} {M(n)}$ leads to $G^{{L_1}R}\leq_\mathcal{E} G$ (Lemma~\ref{lem:dis}),
  which is a contradiction because $G^{L_1}$ is not reversible.

If we have the second inequality, $G^{L_2}+H^{R_1}\geq_\mathcal{E}  0$, then we can assume  $G^{L_2}$ is not end-reversible; if it was end-reversible, then as above, $H^{R_1}$ would be end-reversible and so $G^{L_1}$ would also be end-reversible, contradicting our assumption.
By induction, 
$G^{L_2}+H^{R_1}\equiv_\mathcal{E}  0$ implies that  $H^{R_1}=-G^{L_2}$, since $G$ and $H$ are in reduced form.
But now we have $0\succ_\mathcal{E}  G^{L_1}+H^{R_1} =  G^{L_1}-G^{L_2}$, or $G^{L_2}\succ_\mathcal{E}  G^{L_1}$, which is a contradiction because $\GL$ should have no dominated options.
Therefore, $G^{L_2}+H^{R_1}\succ_\mathcal{E}  0$.
By Claim (i) and the above argument, there must exist a Right option $H^{R_2}$ in $H$ and a Left option $G^{L_3}$ in $G$, such that  
$G^{L_2}+H^{R_2}\leq_\mathcal{E}  0$ and $G^{L_3}+H^{R_2}\succ_\mathcal{E}  0$, and from there a $H^{R_3}$ and $G^{L_4}$, and so on. We get the following chain of inequalities:
	\[
      \begin{array}{ccc}
      G^{L_1}+H^{R_1}\leq_\mathcal{E}  0, &\quad& G^{L_2}+H^{R_1}\succ_\mathcal{E}  0\\
      G^{L_2}+H^{R_2}\leq_\mathcal{E}  0, &\quad&  G^{L_3}+H^{R_2}\succ_\mathcal{E}  0\\
       \vdots &&\vdots
      \end{array}
      \]
But the number of Left options of $G$  is finite; at some point, we will get an inequality like  $G^{L_{1}}+H^{R_m}\succ_\mathcal{E}  0$  (re-indexing if necessary). 

Because the inequalities are strict, summing the left-hand and the right-hand inequalities gives, respectively,
\[ \sum_{i=1}^m G^{L_i}+\sum_{i=1}^m H^{R_i}\leq_\mathcal{E}  0\quad\mbox{ and } \quad
 \sum_{i=1}^m G^{L_i}+\sum_{i=1}^m H^{R_i}\succ_\mathcal{E}  0\]
 which is a contradiction.

 Therefore, we conclude that $G^{L}+H^{R}\equiv_\mathcal{E}  0$, which proves the claim, and then by induction
 $H^R\equiv_\mathcal{E} -G^L$, which solves this case.\\

\noindent {\bf Case 2:}  $G^L$ is end-reversible; that is,  $G^L=\left\{ \; \mid M(n)\right\}$.

Note that any other options of $G$ would be non-reversible. 
We will argue that $G^L$ must be paired with a symmetric end-reversible option in $\HR$. If not, by Case 1, $G+H$ would be
$$\left\{\left\{ \; \mid M(n)\right\},G^{L_2},G^{L_3},\ldots\mid\ldots\right\}+\left\{\ldots\mid-G^{L_2},-G^{L_3},\ldots\right\}.$$

Since $G^L$ is end-reversible, we know $G\geq_\mathcal{E} M(n)$. Adding $H$ to both sides and using $G+H\equiv_\mathcal{E} 0$, we get 
$0\geq_\mathcal{E} H+M(n)$. 

Let us first assume that $\GL$ has more than one option, so the end-reversible option $G^L$ must be fundamental. This means there is a Left-end $X$ such that $\left\{ \; \mid M(n)\right\}+X$ is the only winning move for left in $G+X$. As ends, both $M(n)$ and $X$ are invertible, and so we can add $\Conj{M(n)}+\Conj{X}$ to both sides of $0\geq_\mathcal{E} H+M(n)$ to see $\Conj{M(n)}+\Conj{X} \geq_\mathcal{E} H+\Conj{X}$. Now, Right wins first on $\Conj{M(n)}+\Conj{X}$ (Right end), so Right must have a good first move on $H+\Conj{X}$. But then Left would have a good first move on $\Conj{H}+X$, which contradicts that $G^L$ is fundamental in $G$. Thus, $H$ must have an end-reversible Right option after all, say $\left\{ \; \mid M(n')\right\}$, and since $G$ and $H$ are reduced it must be that $n=n'$.

All that remains is the case where $G^L=\left\{ \; \mid M(n)\right\}\in\GL$ is the only Left option. If so, if $\HR$ is not empty, the proof follows as above. If $\HR= \emptyset $, by Theorem \ref{thm:endinverses}, $\GL$ should be empty, andistence of $G^L$.\\

We have seen that each $G^L$ has a corresponding $-G^L$ in the set of Right options of $H$. This finishes the proof.
\end{proof}

We have the same result for dicot games.

\begin{theorem}\label{thm:ConjugateD} The universe $\mathcal{D}$ of dicot games has the conjugate property; that is, for all 
$G,H\in \mathcal{D}$, if $G+H\equiv_\mathcal{D} 0$ then $H\equiv_\mathcal{D} \Conj{G}$.
\end{theorem}

\begin{proof}
Let $G,H\in \mathcal{D}$ satisfy $G+H\equiv_\mathcal{D} 0$, and consider $G,H$ in their (unique) reduced, canonical forms.
We will prove, by induction on the rank of $G+H$, that we must have $H\equiv_\mathcal{D} \Conj{G}$. 
As before, if we know the conjugate is an additive inverse, we will write
$-G$ for $\Conj{G}$.
For the base case, if $G$ and $H$ both reduce to $0$, then trivially $H\equiv_\mathcal{E} \Conj{G}$.
Otherwise, the game $G+H\in \mathcal{D}$ has at least one Left option and at least one Right option. Call the Left option $G^L$.
 We break the proof into two cases based on whether or not $G^L$ is end-reversible.
\\

\noindent {\bf Case 1:}  $G^L$ is not end-reversible.

Note then that $G^L$ is not reversible at all, because $G$ is in reduced form. 
Claim (i) from the proof of Theorem \ref{Conjugate} follows as in $\mathcal{E}$, this time by Theorem \ref{thm:comparisond}, and so we know there exists $H^{R}$ such that $G^{L}+H^{R}\leq_\mathcal{D} 0$. Claim (ii) from the previous proof says that in fact $G^L+H^R\equiv_\mathcal{D} 0$. This also follows as in $\mathcal{E}$ (except the only end-reversible option $H^R$ is $*$). Together the claims imply that there is a non-reversible option $H^R$ with $G^L+H^R\equiv_\mathcal{D} 0$, and by induction this means $H^R\equiv_\mathcal{D}-G^L$.
\\

\noindent {\bf Case 2:}  $G^L$ is end-reversible; that is,  $G^L=*=\left\{0\mid 0\right\}$.

Since $*$ is reversible in $G$, we know $G\geq_\mathcal{D} 0$. Adding $H$ to each side gives $G+H\geq_\mathcal{D} H $, and since $G+H\equiv_\mathcal{D}0$, this shows $0\geq_\mathcal{D} H$.

We will show by contradiction that $H$ must have a Right option to $*$. Assume $H$ has no Right option to $*$; by Case 1, if $G= \left\{*,G^{L_2},G^{L_3},\ldots\mid\ldots\right\}$ then $H=\left\{\ldots\mid-G^{L_2},-G^{L_3},\ldots\right\}$. 
By Theorem \ref{thm:atomicd} (1), since $G$ is reduced and $*$ is an end-reversible option of $G$, no $G^{L_i}$ can be a winning left move in $G$. By symmetry, then, $H$ has no winning right move. But this contradicts $0\geq_\mathcal{D} H$, as Right wins first on $0$.
Thus, if $G^L$ is end-reversible, then there is a corresponding end-reversible $H^R$ in $H$, and both are $*$. Since $-*=*$, we have $H^R\equiv_\mathcal{D}-G^L$.
\\

So we have pairwise equivalence of the options of $G$ and $H$, and so $G\equiv_\mathcal{D} H$.
\end{proof}

Thus, we have no non-conjugal inverses modulo dicots or dead-ends, and in fact this tells us the same is true for the dicot subuniverse of impartial games, $\mathcal{I}$. The corollary follows from a result of \cite{Renault}, that for $G,H\in \mathcal{I}$, $G\equiv_\mathcal{D}H$ if and only if $G\equiv_\mathcal{I}H$.

\begin{corollary}
The universe of impartial games has the conjugate property.
\end{corollary}


\section{Summary}
We have adapted general results from absolute CGT \cite{LarssonNPS} for mis\`ere play, in order to develop subordinate comparison (Section 2) and reversibility through ends (Section 3.1) for dicot and dead-ending games. We have used these tools to show our two main  results: $\mathcal{D}$ and $\mathcal{E}$ have unique reduced forms  (Section 3.2) and have the conjugate property (Section 4). 

Future work in this area would see the results here extended to other universes besides $\mathcal{D}$ and $\mathcal{E}$. That is, can we get a subordinate comparison test in other universes? Can we solve reversibility through ends, and thus find  canonical  forms? Can we classify which universes have the conjugate property? 

It will also be interesting to see our results for the dead-ending universe applied to specific rule sets within $\mathcal{E}$, such as domineering, in order to solve such games under mis\`ere play.

\end{document}